%% file: mainzpn.tex
\documentclass[10pt,amssymb,amsfonts,psfig]{amsart}
\usepackage[dvips]{epsfig}
\usepackage{amsfonts,amssymb,color,amsmath,amscd,}
\usepackage{graphicx,} 
\usepackage[arrow,matrix]{xy}

\input{macros.tex}
\begin{document}

\bigskip\bigskip

\title[On cycles and coverings associated to a knot]
{On cycles and coverings associated to a knot}


\author {Lilya Lyubich and Mikhail Lyubich.}
\date{\today}

\begin{abstract}  
Let $ \mathcal{K} $ be a knot, $G$ be the knot group, $K$ be its commutator 
subgroup,  
and $x$ be a distinguished meridian.
Let $\Sigma $ be a finite abelian group.
The dynamical system introduced by D. Silver and S. Williams in [S],[SW1] 
consisting of the set 
$\Hom (K, \Sigma )$ of all representations $\rho : K \rightarrow \Sigma$
endowed with the weak topology, together with the homeomorhpism
$$
\sigma _x : \Hom (K, \Sigma) \longrightarrow \Hom (K,\Sigma); \;
\sigma _x\rho (a)=\rho (xax^{-1}) \; \forall a \in K, \rho \in \Hom (K, \Sigma)
$$
is finite, i.e. it consists of several cycles.
In \cite{L} we found the lengths of these cycles   
for $ \Sigma = \Z /p, \; p $ is prime, in terms of the roots of the Alexander
polynomial of the knot, $\mod p.$
In this paper we generalize this result to a general abelian group $ \Sigma $. 
This gives a complete classification of depth 2 solvable coverings 
over $ S^3 \backslash  \mathcal{K} $.

\end{abstract}  

\setcounter{tocdepth}{1}
 
\maketitle
\tableofcontents

\input{intro.tex}

\input{one.tex}
\input{linearalgebra.tex}

\input{mainresult.tex}

\input{genresult.tex}

\input{pullback.tex}

\input{coveringsl.tex}

\input{padic.tex}

\input{bib.tex}
\end{document}

%% file: macros.tex
\newtheorem{thm}{Theorem}[section]

\newtheorem{prop}[thm]{Proposition}

\theoremstyle{remark}

\theoremstyle{definition}

\newtheorem*{LY Theorem}{Lee-Yang Theorem}

\numberwithin{equation}{section}
\numberwithin{figure}{section}

\font\nt=cmr7

\def\note#1
{\marginpar
{\nt $\leftarrow$
\par
\hfuzz=20pt \hbadness=9000 \hyphenpenalty=-100 \exhyphenpenalty=-100
\pretolerance=-1 \tolerance=9999 \doublehyphendemerits=-100000
\finalhyphendemerits=-100000 \baselineskip=6pt
#1}\hfuzz=1pt}

\newcommand{\ra}{\rightarrow}

\renewcommand{\mod}{\operatorname{mod}}

\newcommand{\tl}{\tilde}

\newcommand{\Hom}{\operatorname{Hom}}

\newcommand{\id}{\operatorname{id}}

\newcommand{\isom}{\approx}

\newcommand{\eps}{{\varepsilon}}

\newcommand{\si}{{\sigma}}
\newcommand{\Si}{{\Sigma}}

\newcommand{\BB}{{\mathcal B}}
\newcommand{\CC}{{\mathcal C}}

\newcommand{\KK}{{\mathcal K}}

\newcommand{\SSS}{{\mathcal S}}

\newcommand{\VV}{{\mathcal V}}

\newcommand{\XX}{{\mathcal X}}
\newcommand{\YY}{{\mathcal Y}}

\newcommand{\N}{{\Bbb N}}

\newcommand{\Z}{{\Bbb Z}}

\def\B0{{\mathbf{0}}}

\catcode`\@=12

\def\Empty{}
\newcommand\oplabel[1]{
  \def\OpArg{#1} \ifx \OpArg\Empty {} \else
  	\label{#1}
  \fi}
		
\newcommand{\comm}[1]{}
\newcommand{\comment}[1]{}

%% file: intro.tex
\section{Introduction}
Let $ \KK $ be a knot, $X$ be the knot complement in $S^3$ ,   
$ X= S^3 \setminus \KK $,  $ X_\infty $  be the infinite cyclic cover of $ X $, 
and  $ X_d $ be the cyclic cover of $X$ of degree $d$.

Let $G$ be the knot group, $K$ be its commutator subgroup, and  
$\Sigma $ be a finite group. Let $x$ be a distinguished meridian 
of the knot.
The  dynamical system introduced by D. Silver and S. Williams 
in \cite{Si} and \cite{SW 1} consisting of the set 
$\Hom (K, \Sigma )$ of all representations $\rho : K \rightarrow \Sigma$
endowed with the weak topology, together with the homeomorhpism
$\sigma _x $ (the shift map):
$$
\sigma _x : \Hom (K, \Sigma) \longrightarrow \Hom (K,\Sigma); \;
\sigma _x\rho (a)=\rho (xax^{-1}) \; \forall a \in K, \rho \in \Hom (K, \Sigma).
$$
is a shift of finite type (\cite{SW 1}). 
Moreover, if $ \Sigma $ is  abelian, 
this  dynamical system  is finite, i.e. it  consists of several cycles 
(\cite{SW 2},\cite{K}).  
In (\cite{L}) we calculated the lengths of these cycles and their lcm 
(least common multiple) for $\Sigma =\Z /p $, $ p $ prime, 
in terms of the roots of the Alexander polynomial of 
the knot, $ \mod p $.
Our goal is to generalize these results to an arbitrary finite abelian group 
$\Sigma $. This gives a complete classification of 
solvable depth  2  coverings of
$  S^3 \setminus \KK $. (By a solvable covering of depth $n$  we mean 
a composition of $n$ regular coverings $M_0 \ra M_1 \ra \ldots \ra M_n $
with corresponding groups $\Gamma _i $, such that 
$ \Gamma _0 \lhd \Gamma _1 \lhd \ldots \lhd \Gamma _n $ and
$\Gamma _{i+1}/\Gamma _i $ is abelian.)

Let $ \Delta (t) =c_0+c_1(t)+ \ldots + c_nt^n $ 
be the Alexander polynomial of the knot $ \KK $, and $ B-tA $ its Alexander
matrix of size, say, $ m \times m$, corresponding to the Wirtinger presentation.
From \cite{L} we know that 
\begin{equation}
\Hom (K,\Z /p) \cong (\Z /p)^n \;\; \text{ where }
n=\deg (\Delta (t)\; \mod p).
\end{equation}
It turns out that the same result is true for a target group $ \Z /p^r $:
\begin{equation}\label{MR}
\Hom (K,\Z /p^r) \cong (\Z /p^r)^n \;\; \text{ where }
n=\deg (\Delta (t)\; \mod p).
\end{equation}
In section 2 we give a proof of \eqref{MR} for two-bridge knots. 
In section 3 we prove a general result
about solutions of the reccurence equation
\begin{equation}\label{RE}
  Bx_j-Ax_{j+1}=0,
\end{equation}
where $ x_i \in \XX ,\;  \XX \text{ and }\YY $ are  finite modules, and
$A,\:B:\XX  \ra \YY $ are module homomorphisms.
We then use this result in section 4 to prove  \eqref{MR}
for an arbitrary knot. In section 5 we describe the set of periods and 
calculate their lcm for target group $\Sigma = \Z /p^r $, based on similar 
results for the target group $ \Z /p $, obtained in \cite{L}. 
We then generalize these results for any finite abelian group $\Sigma $.

In section 6 we describe the relation between 
the shift $\sigma _x $ on $\Hom (K,\Sigma ) $ and the pullback map $\tau ^* $ 
 corresponding to the meridian 
$ x $,
on the space of regular coverings over $ X_\infty $ . 
In section 7 we construct a regular covering $ p: N \ra X_d $
with the group of deck transformations $\Sigma $, corresponding to a surjective 
homomorphism $ \rho \in \Hom(K,\Sigma )$ with $\sigma ^d_x \rho =\rho $,
and prove that any regular covering of $ X_d $ with the group of deck 
transformations $ \Sigma $ can be obtained in this way. We conclude the paper 
by formulating our results in terms of $p-$adic representations of $K$ 
and associated 
solenoids and flat principal bundles.

%% file: one.tex
\section{Case of a two-bridge knot}
Let $\Delta (t) $ be the Alexander polynomial of a two-bridge knot 
$\mathcal {K} $ and 
$ n $ be the degree of $\Delta (t)\; \mod p.$ 
Since the Alexander polynomial 
is defined up to multiplication by $ t^k, k \in \Z ,$ and has symmetric 
coeffitients, we can write
$$
\Delta (t)=pd_kt^{-k}+\ldots +pd_1t^{-1} +c_0+c_1t+ \ldots +c_nt^n+pd_1t^{n+1}+
\ldots pd_kt^{n+k}t^{n+k}, 
$$ 
where $ c_i, d_i $ are integers and $ c_0=c_n $ 
is not divisible by $ p.$
Similarly to the Theorem 9.1 in \cite{L} we can prove that
$ \Hom (K,\Z /p^r) $ is isomorphic to the space of bi-infinite sequences
$ \{x_i \}_{ i \in \Z} ,\; x_i \in \Z /p^r, $ satisfiing  the following 
reccurence equation 
$ \mod p^r $:
\begin{equation}\label{eq1}
pd_kx_{-k+j}+\ldots pd_1x_{-1+j} +c_0x_j+c_1x_{j+1}+ \ldots +c_nx_{n+j}+
\end{equation}
\begin{equation*}
 +pd_1x_{n+1+j}+ \ldots +pd_kx_{n+k+j}=0
\end{equation*}
From \cite{L} we know that 
$\Hom (K,\Z /p) \cong (\Z /p)^n $ where $ n=\deg ( \Delta(t) \; 
\mod p) $. The same is true for target groups $ \Z /p^r $.
\begin{thm}\label{thm1}
$\Hom (K,\Z /p^r) \cong (\Z /p^r)^n $ where $ n=\deg ( \Delta(t) \; 
\mod p).$
\end{thm}
\begin{proof}
We will prove that $ x_0,x_1,\ldots ,x_{n-1} \in \Z /p^r $ uniquely determine the 
sequence $ \{x_i \}_{i \in \Z }, \;x_i \in \Z /p^r $, satisfiing equation 
\eqref{eq1}.
The proof is by induction. For $ r=1, $ given  $ x_0, x_1,\ldots, x_{n-1}
\in \Z /p,\;\; x_n  $ is uniquely determined $\mod p $ by the equation 
\begin{equation}\label{eq0}
c_0x_0+c_1x_1+\ldots +c_nx_n=0 \; (\mod p).
\end{equation}
So, 
 $ x_0, x_1,\ldots, x_{n-1} \mod p $ uniquely determine
 the whole sequence $ \{x_i\}_{i \in \Z} \;$  
$ (\mod p), $ satisfiing \eqref{eq1}.
This proves the base of induction.

Suppose the statement is true for $ r $. 
Fix  $  x_0, x_1,\ldots, x_{n-1} \;\mod p^{r+1}$  
and let  $ \{x_i \}_{i \in \Z} $ be the sequence satisfiing equation: 
\begin{equation}\label{0}
pd_kx_{-k}+\ldots + pd_1x_{-1} +c_0x_0+c_1x_{1}+ \ldots +c_nx_{n}
+ \ldots +pd_kx_{n+k}=0 \;
\mod p^r. 
\end{equation}
It is uniquely determined $ \mod p^r $, by induction assumption.  
But then all the terms of 
\eqref{0} 
except $c_nx_n$ are determined $ \mod p^{r+1}$. So $x_n $  
and hence 
the whole sequence $ \{x_i\}_{i \in \Z} $ is uniquely determined 
$\mod p^{r+1}$ 
by $  x_0, x_1,\ldots, x_{n-1} \; \mod p^{r+1}. $ 
\end{proof}

%% file: linearalgebra.tex
\section{Linear matrix reccurence equations}
\begin{thm}\label{LA}
Let $\: \XX,\YY \:$ be two finite modules of the same order, 
over the same ring $\: R. $
Let $\: A,B : \XX \longrightarrow \YY \: $ be modules homomorphisms such that 
$\: \ker A
\cap \ker B =0. \:$  Consider the following reccurence equation:
\begin{equation}\label{*}
Bx_j-Ax_{j+1}=0
\end{equation}
Then $\:\XX=\mathcal{V} \oplus \mathcal{A} \oplus \mathcal{B},\: $ where
$ \;\mathcal{V}=\{v \in \XX \;: \; $  there exists a bi-infinite sequence 
$\ldots v_{-1}, v_0=v, v_1, \ldots ,$ satisfiing equation \eqref{*},  \}
 
\noindent
$\mathcal{A}=\{a \in \XX \;:\;$ there exists an infinite 
sequence $\:\ldots,
a_{-1},a_0 \;$ satisfiing \eqref{*} and $\;a_{-i}=0 \;$ 
for sufficiently large $\: i $ \}. 

\noindent
$\mathcal{B}=\{b \in \XX \;:\;$ there exists an infinite sequence
$ \;b_0=b, b_1, b_2,
\ldots,\: $ satisfying\eqref{*} 
and $ \;b_i =0 \;$ for sufficiently large $\: i .\}$ 
\end{thm} 
\begin{proof}
The proof is   by induction in the order of $ \XX $ and $ \YY $. 
Consider a diagram :
$$
\xymatrix{
& \XX \ar[dl]_{A}
    \ar[d]^{\pi _1}
    \ar[dr]^{B} \\
\YY \ar[dr]_{\pi _2} &
\XX /\ker A \ar @<-2pt>[d]_{\bar{A}}
         \ar @<2pt>[d]^{\bar{B}} &
\YY \ar[dl]^{\pi _2} \\
& \;\;\;\;\; \YY /B(\ker A) & }
$$
where by definition,
$ \pi _1 $ and $ \pi _2 $ are factorization maps;
$\; [x]= \pi _1(x); \;$ and
$$
\bar{\mathcal{A}}([x])=\pi _2\circ A (x),\;\;\;
\bar{\mathcal{B}}([x])=\pi _2\circ B (x).
$$

\noindent
This diagram is not commutative, but its
left- and right-hand triangles are commutative.
Note that $\: \XX / \ker A \:$ and $\: \YY / B(\ker A)\:$ 
are modules over $ \:R \: $ 
of the same order,
since $ B $ is injective on $ \ker A $.
   
Suppose that the statement of the theorem is true for $ \XX / \ker A \;$ 
and operators $\: \bar{A}\: $ and $\: \bar{B}:$
\begin{equation}
 \XX /\ker A = \bar{\mathcal{V}}  \oplus 
\bar{\mathcal{A}}  \oplus \bar{\mathcal{B}}, 
\end{equation}
where all the  sequences in definition of $\; \bar{\mathcal{V}},\:
\bar{\mathcal{A}},\: \bar{\mathcal{B}} \:$
satisfy 
the equation:
\begin{equation}\label{**}
\bar{B}[x]_i- \bar{A}[x]_{i+1}=[0].
\end{equation}
Then we will prove that
\begin{equation}
\XX =\mathcal{V} \oplus \mathcal{A} \oplus \mathcal{B}, 
\end{equation} 
Take any $ \:u \in \XX .\: $
By induction assumption $\;[u]=[v]+[a]+[b]\;$, where
$\;[v] \in \bar{\mathcal{V}},\; [a] \in \bar{\mathcal{A}},\;  
[b] \in \bar{\mathcal{B}}. \;$ We find lifts $\; v,a,b \;$ of 
$\; [v],\: [a],\:[b] \;$ to $\; \mathcal{V},\: \mathcal{A},\: \mathcal{B}\; $ 
respectively.  
Let $\: \ldots , [v_{-1}], [v_0]=[v],[v_1], \ldots \: $ satisfy $ \:\bar{B}[v_i]-
\bar{A}[v_{i+1}]=[0],\;\; i \in \Z  .$
Take any lift $ \ldots, y_{-1}, y_0, y_1, \ldots .$ Then $\: By_i-Ay_{i+1}=x_i
 \in B(\ker A).\; $ So $\: x_i=Bw_i\: $ for some $\; w_i \in \ker A.\;$ 
Then $$ B(y_i-w_i)-A(y_{i+1}-w_{i+1})=0. \;$$
So $\; v_i=y_i-w_i \;$ satisfy \eqref{*} and $ \;v=v_0 \in \mathcal{V}\; $ 
is a desired lift of $ \;[v].\;$

Similarly, for $\; [a] \in \bar{\mathcal{A}}\; $ there exists a sequence
$\; \ldots , [a]_{-1}, [a_0]=[a] ,\; $ satisfying\eqref{**} with 
$\;[a]_{-i}=[0] \;$ for $\: i \ge N \: $.  As before, we can find a lift  
$\; \{ a_{-i}\}_ {  i\ge 0} ,\: $  satisfying$\; Ba_{-i}- Aa_{-(i-1)}=0.\; $ 
Note that $\: a_{-i} \in \ker A \:$ for $\: i \ge N. \:$ 
We have $$B \cdot 0=Aa_{-N} .$$ 
But then the sequence 
$\; \ldots , 0, 0, a_{-N}, a_{-(N-1)}, \ldots, a_0 \;$ 
also  satisfies \eqref{*}, 
so $\; a=a_0 \in \mathcal{A} \;$ is a desired lifting. 

We repeat the same argument to  prove that $ \;[b]\; $
has a lift $\; b \in \mathcal{B}.\; $  
If $\;\{ [b_i]\} _{i \ge 0}\: $ satisfies \eqref{**} and $\; [b_i]=0 \;$ 
for $ \; i \ge N,\; $ 
we find a lift $ \;  \{ b_i \} _{i \ge 0}\: $ satisfying\eqref{*}.
Since $\; b_i \in \ker A \; $ for $\;i \ge N,\; $
and $\; Bb_i-Ab_{i+1}=0 ,\; $ we have  also $\;b_i \in \ker B \;
 \text{ for } i \ge N-1,\; $ 
hence $\; b_i=0 \text{ for } i\ge N-1,\; $ since by assumption $\;\ker A \cap 
\ker B = 0. \;$
So $\;  b=b_0 \in \mathcal{B}\; $ 
is a desired lift. Since $\; \pi _1(u)=\pi _1(v+a+b), \;
u= v+a+b+\tilde{a},\;$ where 
$\;\tilde{a} \in \ker A \;$ and so $\; \tilde{a} \in \mathcal{A}.\;$
The step of induction is done.  

Since we can interchange the roles of $ A $ and $ B $, it remains to prove 
the statement of the theorem in the case when $ A $ and  $ B $ are monomorphisms
and hence are isomorphisms, since $ \;|\XX |=|\YY |.\;$ In this case any 
element 
$ x \in \XX \;$
has a  bi-infinite continuation $ \; x_i=(A^{-1}B)^ix, \;$ satisfying\eqref{*}.
The theorem is proven.
\end{proof}

%% file: mainresult.tex
\section{Main result for a general knot}
In this section we prove that the Theorem \ref{thm1}  holds for 
any knot.
Let $ B-tA $ be the Alexander matrix of a general knot $ \KK $ arising from
the Wirtinger presentation of the knot group $ G $. Here $ A,B  $ are 
$ m\times m$
matrices with elements $ 0, \pm 1 $. 
\begin{thm}
Dynamical system $ ( \Hom (K, \Sigma ), \sigma _x) $ is conjugate to the left 
shift in the space of bi-infinite sequences 
$ \{y_j\}_{j\in \Z },\: y_j \in (\Sigma )^m $
satisfyingreccurence equation
\begin{equation}\label{BA}
By_j-Ay_{j+1}=0. 
\end{equation}
\end{thm}
For  the target group $ \Z /p $
this result is proven  in \cite{L}, Theorem 4.2.
For a general abelian group $ \Sigma $ 
the proof is identical .

We can apply theorem (\ref{LA}) for modules $ (\Z /p^r)^m $ 
and linear operators 
$A,B : (\Z /p^r)^m \ra  (\Z /p^r)^m $  given by matrices $ A $ and $ B $ 
to get
\begin{equation}\label{oplus}
(\Z/p^r)^m=\VV_r \oplus \mathcal{A}_r \oplus \BB_r ,
\end{equation}
where $\VV_r =\{ y \in (\Z/p^r)^m  \; :$  there exists 
a bi-infinite sequence
 $\;\;\ldots , y_{-1},y_0=y, y_1,\ldots , \;\;$ 
satisfying equation \eqref{BA}$\}$, 

\noindent
$ \mathcal{A}_r=\{a \in (\Z/p^r)^m \; :$  there exists an 
infinite sequence 
$\;\ldots ,a_{-1},a_0=a \;$, satisfying \eqref{BA} and $a_{-i}=0 \;$ for 
sufficiently large $ i\; \}$,

\noindent
$\BB_r =\{b \in (\Z/p^r)^m \; :$  there exists an infinite sequence 
$ \;\;b=b_0, b_1, b_2,\ldots,\;\; $ satisfying \eqref{BA} and $ b_i=0 $ 
for sufficiently 
 large $ i \;\}$.
 
We will use the uniquiness of continuatuion that follows from the finiteness 
of $ \Hom(K,\Sigma ) $ for a finite abelian group $ \Sigma $
(see   Proposition 3.7 \cite{SW 2} and Theorem 1 (ii) \cite{K}). 
If $ \{x_i\}_{i \in \Z } \; $ and $ \{y_i\}_{i \in \Z } \; $ satisfy \eqref{BA},
then $ x_0=y_0 $ implies $ x_i=y_i \; \forall \:i.$ In particular, 
for $ a \in \mathcal{A}_r,\; a \ne 0 ,\;$ there is no infinite continuation 
to the right, satisfying \eqref{BA},  
and for $ b \in \mathcal{B}_r,\; b \ne 0 ,\;$ there is no 
infinite continuation to the left, satisfying \eqref{BA}.
(Otherwise we would have  two bi-infinite sequences: 
$\ldots , 0, 0, \ldots, a_0, a_1, \ldots $ and $ \ldots , 0, 0, \ldots $.)
So $ \Hom (K, \Z /p^r) $ being isomorphic to the space of be-infinite sequences
satisfying \eqref{BA}, is isomorphic to $ \mathcal{V}_r .$

Since the only decomposition of $ (\Z /p^r)^m $ as a direct sum of three 
groups is 
$$
 (\Z /p^r)^m\cong (\Z /p^r)^{n_r} \oplus (\Z /p^r)^{l_r} \oplus (\Z /p^r)^{m_r} 
\text{ with } n_r+l_r+m_r=m,
$$
 it follows from \eqref{oplus} that $\mathcal{V}_r\cong (\Z /p^r)^{n_r}$. 
Consider the projection: 
$$
\xymatrix{ 
(\Z /p^{r+1})^m  = \mathcal{V}_{r+1} \oplus \mathcal{A}_{r+1} 
\oplus \mathcal{B}_{r+1}\ar @<-40 pt>[d] _{\pi }\\
(\Z /p^r)^m = \mathcal{V}_r \oplus \mathcal{A}_r \oplus \mathcal{B}_r }
$$
Clearly $ \pi (\mathcal{V}_{r+1}) \subset \mathcal{V}_r , \;
 \pi (\mathcal{A}_{r+1}) \subset \mathcal{A}_r , \;
 \pi (\mathcal{B}_{r+1}) \subset \mathcal{B}_r \;.$
It follows that $ n_r $ is the same for all $ r $.
Since from Theorem 5.5 \cite{L} it immediately follows that 
$ n_1 =\deg (\Delta (t) \mod p ),$ we have proven the following theorem:
\begin{thm}\label{genresult1}
For any knot, $ \Hom (K,\Z /p^r) \cong (\Z /p^r)^n,$
where $ n=\deg (\Delta (t) \mod p)$.
\end{thm}

%% file: genresult.tex
\section{Least common multiple}

\begin{prop}
The dynamical system $ (\Hom(K, \Z /p^r), \sigma _x) $ is isomorphic to 
$ (\VV _r,T_r) $, where
$ T_r=(A | \mathcal{V}_r)^{-1}( B | \mathcal{V}_r).\;$
\end{prop}
\begin{proof}
Restrictions $ A | \mathcal{V}_k \; $ and $ B | \mathcal{V}_k \; $
are isomorphisms, since 
$ \ker A \in \mathcal{A}_k $ and $ \ker B \in \mathcal{B}_k.$
Also $ A \mathcal{V}_k = B \mathcal{V}_k $ since every element  
$v \in \mathcal{V}_k $ has continuation to the right and to the left: 
there exist $  v_{-1} $ and $ v_1 $ such that $ Bv_{-1}=Av, \; Bv=Av_1.$
So $ T_r :\VV _r \ra \VV _r $ is well defined, and since $ T_r $ is conjugate 
to the left shift in the space of sequences satisfiing equation\eqref{RE}, 
the formula 
$ T_r=(A | \mathcal{V}_r)^{-1}( B | \mathcal{V}_r)\;$ is obvious.
\end{proof}

In \cite{L} we calculated the set of periods of orbits and their lcm
for dynamical system $ (\Hom (K, \Sigma ), \sigma _x) $ with $ \Sigma =\Z /p $ 
in terms of orders and 
multiplicities of the roots of $ \Delta (t)\; \mod p $.
Now we find the lcm and  the set of periods for $ \Sigma =\Z /p^r $.
\begin{thm}\label{lcm}
Let $d_r= $ lcm  of periods of orbits of $ (\Hom (K,\Z /p^r), \sigma _x) $.
Then either $d_i=d_1 \;\forall i $, or
$\exists \:s\ge 1 $ such that $ d_1= \ldots =d_s, $ and $ d_{s+i}=d_1p^i $.
\end{thm}
\begin{proof}
The following diagram commutes:
$$
\xymatrix{
\ldots  \ar [r]^{\pi } & \VV _{k+1} \ar [d]_{T_{k+1}}
                                   \ar [r]^{\pi } &
\VV _k \ar [d]_{T_k} \ar [r]^{\pi }  &
\;\ldots \;\ar [r]^{\pi } & 
\VV _1 \ar [d]^{T_1} \\ 
\ldots \;\ar [r]^{\pi } & \VV _{k+1} \ar [r]^{\pi } &
\VV _k  \ar [r]^{\pi } &
\;\ldots \; \ar [r]^{\pi } & 
\VV _1 
}
$$
Let $\VV =\underleftarrow{\lim } \VV _k , \; \VV_k \subset (\Z _p)^m $,
where $ \Z _p $ is the set of $p-$adic numbers, and $ T:\VV \ra \VV ,$ 
$ T=\underleftarrow{\lim }T_k $.
 We will use the same notations  for module homorphisms 
 and their matrices in the standard basis.
Let $ E_r, \; E $ denote the identity isomorphisms of  $  (\Z /p^r)^n $ and
$ (\Z _p)^n $ respectively.
We have $ T_1^{d_1}=E_1 $, so either $ T^{d_1}=E $, and then 
$ T_r^{d_1}=E_r \;\forall r $, 
or $T^{d_1}=E+p^sA $ for some $  s\in \Z ,\;s\ge 1 $, 
and not all elements of matrix $ A $ are divisible by $ p $.
In the later case $ T_i^{d_1}=E_i ,\;i=1,\ldots ,s $. Since 
$$
T^{d_1\cdot k}=(E+p^sA)^k=E+kp^sA+C_k^2p^{2s}A^2+\ldots + p^{s\cdot k}A^k,
$$
we have $ T^{d_1p}=E+p^{s+1}A_1 $, where not all elements of $ A_1 $ 
are divisible by $ p $,
and, by induction, $ T^{d_1p^i}=E+p^{s+i}A_i ,\; \forall \:i\ge 1 $, 
where not all elements of $ A_i  $ 
are divisible 
by $ p $. Then $ T_{s+i}^{d_1p^i}=E_{s+i} $  
and the statement of the theorem follows.
\end{proof}
\begin{prop}\label{periods}
Let $ Q_r \subset \N $ be the set of all periods of 
$ (\Hom(K,\Z /p^r), \sigma _x) $ .
Then $ Q_r \subset Q_{r+1}.$
\end{prop}
\begin{proof} 
If $ \{x_j\}_{j \in \Z},\; x_j \in \Z /p^r $ is a sequence satisfiing
reccurence equation \eqref{BA}  $ \mod p^r $ with period $ d, $ then 
 $ \{px_j\}_{j \in \Z},\; px_j \in \Z /p^{r+1} $ satisfies \eqref{BA} 
$ \mod p^{r+1} $ and
has the same period.
\end{proof}

Now we turn to a general finite abelian group $\Sigma $, which is isomorphic to 
a direct sum of cyclic groups:
$$
\Sigma =\bigoplus _{i\in I} \Z /p_i^{r_i}, \; I\subset \N.
$$
Then 
$$
\Hom (K,\Sigma )=\bigoplus _{i\in I}\Hom (K,\Z /p_i^{r_i})=
\bigoplus _{i\in I} (\Z /p_i^{r_i})^{n_i},\; \text{ where } 
n_i=\deg (\Delta (t) \mod p_i),
$$
and the original dynamical system is the product of dynamical systems:
$$
(\Hom (K,\Sigma ),\sigma _x)=\bigoplus _{i\in I}
(\Hom (K,\Z /p_i^{r_i} ),\sigma _x).
$$
Taking sums of orbits with different periods,
we obtain the following proposition:
\begin{prop}
(i) Let $ d_i $ be lcm of periods of orbits of 
$ (\Hom (K,\Z /p_i^{r_i} ),\sigma _x).$
Then lcm of periods of orbits of $  (\Hom (K,\Sigma ),\sigma _x) $ is 
lcm$\{d_i, i\in I\}.$

\noindent
(ii) Let $ Q_i $ be the set of periods of orbits of  
$ (\Hom (K,\Z /p_i^{r_i} ),\sigma _x).$
Then the set of periods for  $  (\Hom (K,\Sigma ),\sigma _x) $ is
$$
Q=\{lcm\{q_i, i\in I\}, \; q_i\in Q_i\}.
$$
\end{prop}

%% file: pullback.tex
\section{Pullback $\tau^*$ on the space of coverings over $X_\infty$}
Let  $ p_{\infty } : X_{\infty } \longrightarrow X  $ be the infinite
cyclic covering over the complement of the knot, 
and let $\tau: X_\infty\ra X_\infty$ be the deck
tansformation corresponding to the loop $x$. 
We will now give a geometric description of the transformation $\si_x$
earlier defined algebraicly. 

Let us remind the pullback construction.
Let $ P:\: E \ra B $ and $ f:\: Y \ra B $ be two continuous maps.  
$ \Gamma _P=\{(e,b):  
 e \in E,b \in B, P(e)=b \} \subset E \times B $ 
is the graph of $ P $.
We have $ \id \times f :\:E\times Y \ra E \times  B.$
Then, by definition, the pullback of  $ P $ by  $ f $,
$ f^*(P):\:(\id \times f)^{-1}\Gamma _P \ra Y $ is the projection onto the 
second coordinate. We have $(\id \times f)^{-1}\Gamma _P=
\{(e,y):\; e \in E, y \in Y, P(e)=f(y)\}$.  
The projection of this set onto the first coordinate, $\tilde{f}$, 
is the lift of $f$, since  the following diagram commutes:
$$
\xymatrix {
(e,y) \ar[r]^{\tilde{f}} 
      \ar[d]_{f^*(P)} &
e \ar[d]^P  \\
y \ar[r]^f & f(y)=P(e)
}
$$
Note that if $ P $ is a (regular) covering then  
so is $f^*(P)$.

Let $ a \in X_\infty , \; p_\infty(a) = x(0) $ and let $ p:(M,y) \ra (X_\infty ,a) $
be the covering corresponding to a group $ \Gamma \subset \pi _1(X_\infty,a) $,
so that $ p_*(\pi _1(M,y))=\Gamma $.
Let $ p':(M',y') \ra (X_\infty , \tau ^{-1}a) $
be the pull back of $ p $ by $  \tau $. It is a
covering corresponding to the group $ \tau _*^{-1} \Gamma 
\subset \pi _1(X_\infty, \tau ^{-1} a) $.
Then $ \tau :X_\infty \ra X_\infty $ lifts to a homeomorphism 
$\hat{\tau }:M' \ra M $
such that $ p \circ \hat{\tau }=\tau \circ p' $.
$$
\xymatrix {
(M',y') \ar[r]^{\hat{\tau }} 
   \ar[d]_{p'}  &
(M,y) \ar[d]^p \\
(X_\infty , \tau ^{-1}a) \ar[r]^\tau & (X_\infty ,a)
}
$$ 
Let $ \tilde{x} $ be the lift of $ x $ to $ X_\infty $ connecting $ \tau^{-1}a $
to $ a $. If $ \hat{x} $ is the lift of $ \tilde{x} $ to $ M' $ beginning at
$ y' $ and ending at $ y'' $, then   $ p':(M',y'') \ra (X_\infty , a) $ 
is the covering corresponding to the group 
$ \tilde{x} ^{-1} (\tau _*^{-1} \Gamma)\tilde{x} \subset \pi _1(X_\infty,a) $.

Let $ \mathcal{C} $ denote the space of all coverings of $ X_\infty  $ up to
the usual equivalence. Let $ \mathcal{G} $  be the space of conjugacy
classes of subgroups of $  \pi _1(X_\infty,a) \isom K $. There is one-to-one 
correspondance between $ \mathcal{C} $ and $ \mathcal{G} $.
In what follows we will not distinguish notationally between a covering and its 
equivalence class, and between a subgroup and its conjugacy class.

 The pullback transformation 
$ \tau ^*: \mathcal{C} \ra  \mathcal{C} $, 
corresponds to the map $ \tilde{\gamma }: \:
\mathcal{G} \ra \mathcal{G}$, $ \tilde{\gamma }:\:\Gamma \mapsto 
\tilde{x}^{-1}(\tau _*^{-1} \Gamma )\tilde{x} \subset \pi _1(X_\infty,a) $, 
$\forall \: \Gamma \subset \pi _1(X_\infty,a) $, 
which turns into the map $ \gamma $ acting
on the subgroups of $ K \subset \pi _1(X,x(0))$: 
$ \gamma (\Gamma )=x^{-1} \Gamma \:x $, $ \forall \: \Gamma \subset K $.

Regular coverings of $X_\infty $ 
correspond to normal subgroups $ \Gamma \subset K $,
which in turn correspond to representations $ \rho \in \Hom (K,\Sigma ) $
such that $ \ker \rho =\Gamma $, in various groups $ \Sigma $ .
The corresponding map on the space  $  \Hom (K,\Sigma ) $ is $ \sigma _x $,
where $ \sigma _x \rho (\alpha )=\rho (x \alpha x^{-1})$.
Indeed, if $ \Gamma = \ker \rho $, then $ x^{-1}\Gamma x= \ker \sigma _x \rho $.
In summary we can say that {\it the shift $ \sigma _x $ in the space 
$ \Hom (K,\Sigma ) $ defined algebraicly corresponds to the pullback action
of the deck transformation $ \tau $ in the space of regular coverings 
over $ X_\infty $}.

%% file: coveringsl.tex
\section{Coverings of finite degree}
\begin{thm}
There is one-to-one correspondence between the surjective elements
$\rho \in \Hom (K, \Sigma) $ such that $ \sigma _x^d \rho  =\rho $
and regular coverings
$ p : N \rightarrow X_d  $ with the group of deck 
transformations 
$ \Sigma $.
\end{thm}
\begin{proof}   
Let $ \rho $ satisfy the condition of the theorem.
Take a covering $p_\rho : M \ra X_\infty $ corresponding to $ \ker \rho  $.
Since  $ \sigma _x^d \rho  =\rho $, this covering coincides with its
$ d-$time pullback: $ \tau ^{*d} p_\rho =p_{\sigma _x^d \rho } = p_\rho  $.
We can lift $ \tau ^d $ to $ \zeta : M \ra M $ so that the following diagram 
commutes:  
$$
\xymatrix@C=1em{
(M, y') \ar[d]_{p_\rho }
               \ar[rr]^{\zeta } &&
(M, y) \ar[d]^{p_\rho} \\
(X_{\infty }, \tau ^{-d} a) \ar[dr]_{p_{\infty }}
                 \ar[rr]^{\tau ^d} &&
(X_{\infty }, a) \ar[dl]^{p_{\infty }} \\
&( X,x(0)) &
}
$$

If
$ \rho :K \rightarrow \Sigma $ is onto then $ \Sigma \cong K/\ker \rho $ 
acts on $ M $ in the standard way: if $ \alpha \in \pi _1(X_{\infty },a) $ 
is a loop and  $ \tilde{\alpha } $ is its lift to $ M $ 
starting at $ y $, it ends at $ \rho (a)(y) $.
Clearly the action of $ \Sigma  $ commutes with $ \zeta  $.
So $ \Sigma  $ acts on the space  of orbits of $ \zeta ,\;N= M/\zeta $.
These orbits project onto orbits of $ \tau ^d $. 
Since $ X_{\infty }/\tau ^d=X_d $, 
we obtained a regular covering $ p: N \rightarrow X_d $.

   Now we prove that any regular covering over $ X_d $ with the group of deck 
transformations
 $ \Sigma $ can be obtained in this way: namely, for any covering 
(that is convinient to denote by) $ p_2:N \rightarrow X_d $ with $ \Sigma  $
as the group of deck transformations, $ \exists \rho \in \Hom (K, \Sigma ) $
such that $ \sigma _x^d(\rho )=\rho $ and the covering 
$ \eps _2:M \rightarrow X_{\infty } $ corresponding to the subgroup $ \ker \rho $,
such that $ N=M/\zeta  $ , $ \zeta $ being a lift of 
$ \tau ^d $.

\noindent
Consider a diagram 
$$
\xymatrix{
& N \ar[d]^{p_2}  \\
X_{\infty } \ar[r]_{p_1} & X_d
}
$$
where $ p_2 $ is a regular covering with a group of deck transformations
 $ \Sigma $, and $ p_1 $ is an infinite cyclic covering with the generator
$ \tau ^d $.
Let us consider the pullback of $p_2 $ by $p_1$. Let
 $ M \subset N \times X_{\infty },\:
M=\{(a,x)\:|\: p_2a=p_1x \}.$ Then we have two covering maps $ \eps _1 $ and 
$ \eps _2 , \;  \eps _1 (a,x)=a,\; \eps _2(a,x)=x $, 
such that the following diagram commutes: 
$$
\xymatrix{
M \ar[r]^{\eps _1} 
  \ar[d]_{\eps _2} &
N \ar[d]^{p_2}\\
 X_{\infty } \ar[r]_{p_1} & X_d
}
$$
For $ y \in X_{\infty },\: (a_1,y), (a_2,y), \ldots ,(a_s,y) $ are all preimages 
of $ y $ under $ \eps _2 $, where $ a_1, a_2, \ldots ,a_s $ are all preimages of 
$ x=p_1(y) $ under $ p_2 $, and $ (a,y_1), (a,y_2), \ldots , $
are all preimages of $ a \in N $ under $ \eps _1 $, where 
$ y_1, y_2, \ldots $ are all preimages of $ p_2(a) $ under $ p_1 $.
   
Since $ \tau ^d $ is a generator of the group of deck transformations of 
$ p_1 $, $\;\zeta =(\id ,\tau ^d) $ is a generator of the group of deck 
transformations
 of $ \eps _1 $, while $ \{ (\sigma ,\id )|\sigma \in \Sigma \} \cong \Sigma  $
is the group of deck transformations of $ \eps _2 $.

For any $ \beta \in K $ let $ \tilde{\beta } $ be its lift to $ M $ 
starting at 
$ (y_0,\beta (0)) $ and ending at $ (y_1,\beta (0)) $ , 
where $ y_0, y_1 \in N $.
There exists a unique   $  \sigma \in \Sigma $ such that $ \sigma y_0=y_1 $.
Take $ \rho (\beta ) =\sigma $.
It is easy to see that 
$ \beta  \in \ker \rho \text{ iff } x^d(\tau ^d \circ \beta )x^{-d} 
\in \ker \rho $.
So, $  \ker \rho = \ker  \sigma _x^d(\rho ) $.
Since we can think of  $ \rho $ as the homomorphism 
$ \rho :K \rightarrow K/\ker \rho \cong \Sigma $, we have 
$  \sigma _x^d(\rho )= \rho $.
\end{proof}

%% file: padic.tex
\section{$p$-adic solenoids}

The above results can be summarized in terms of solenoids fibered over  
manifolds
$X$ and $X_\infty$. 

Let us have a family  of coverings $p_n: S_n\ra B$, $n=0,1,2\dots$,
over the same $m$-dimensional manifold $B$.
We say that they form a {\it tower} if there is a family of coverings 
$g_n: S_n \ra S_{n-1}$ 
such that $p_n= p_{n-1}\circ g_n$. 
In this case we can form the {\it inverse limit}
$\SSS = \underleftarrow{\lim } \, S_n$ by taking
the space of sequences $\bar z= \{z_n\}_{n=0}^\infty ,\;z_n \in S_n $ such that
$g_n(z_n)= z_{n-1}$. Endow $\SSS $ with the weak topology. 
It makes the natural projection $p_\infty: \SSS \ra B$, 
$  \bar z \mapsto z_0$, a locally trivial fibration
with Cantor fibers (as long as $\deg p_n \to \infty$). Moreover, $\SSS$ has 
a ``horizontal'' structure of   
$m$-dimensional lamination. If it is minimal (i.e., if all the leaves are 
dense in 
$\SSS$), 
it is called a {\it solenoid} over $B$. 

If all the coverings $p_n$ are regular with the  group of deck 
transformations $\Si_n$, then 
$\SSS$ is a flat {\it principal $\Si$-bundle} over $B$ with 
$\Si =\underleftarrow{\lim }\, \Si_n$.  This means
that 

\noindent
(i) $p_\infty: \SSS \ra B $ is a locally trivial fibration with fiber $\Sigma $:
$ \forall b \in B ,\; \exists \: U \subset B, \; U \ni b $ and a homeomorphism 
$\phi _U $ such that the following diagram commutes:   
$$
\xymatrix {
p^{-1}(U) \ar[rr]^{\phi _U} 
          \ar[dr]_{p_\infty} &
&  U \times \Si \ar[dl] \\
& U &
}
$$
(ii) If $ U \cap V \neq \emptyset $ and $ h_{U \cap V} $ is defined by commutative
diagram
$$
\xymatrix {
& p^{-1}(U \cap V) \ar[dl]_{\phi _U} 
                   \ar[dr]^{\phi _V}
 & \\
(U \cap V)\times \Si \ar[rr]^{h_{U,V}} & &
(U \cap V)\times \Si
}
$$
then $\exists \:a=a_{U,V} \in \Si, $ such that $ h_{U,V}(b,\si )=(b,\si +a) $. 

\noindent
In this case $ \Si  $ acts on $\SSS $ preserving fibers, 
so that for all $ \alpha \in \Si $
the following diagram commutes: 
$$
\xymatrix {
p^{-1}(U) \ar[d]_{T_\alpha }
          \ar[r]^{\phi _U } &
U \times \Si \ar[d]^{(b,\si ) \mapsto (b,\si +\alpha )} \\
p^{-1}(U) \ar[r]^{\phi _U} & U \times \Si 
}
$$
(we consider the case of an abelian $\Si $).

Given a  principal flat $\Si$-bundle and a point $b\in B$,
 we can consider  the  monodromy action of $K =\pi _1(B,b)$ on the fiber 
$p^{-1}_\infty(b)$. Each element $\gamma \in K $ acts as  a  translation 
by some $\rho (\gamma ) \in \Si $. (Let us cover the immage of $ \gamma $ by   
neighborhoods $ U_0, U_1,\ldots ,  U_n $ from the definition of flat 
principal $\Si $-bundle, such that $ U_i \cap U_{i+1} \neq \emptyset $,
$ U_n= U_0 $. The monodromy action of $\gamma $ on $ p^{-1}(b) \isom \Si $ 
is the translation by $\rho (\gamma )=\sum _{i=0}^{n-1} \alpha _{U_i,U_{i+1}} $).
This action gives us a representation  $\rho :  K \longrightarrow \Si $.  

Vice versa, given a representation $\rho:  K \ra\Si $,
we can construct a flat principal $\Si$-bundle over $B$ by taking the
{\it suspension} of the $K$-action.  
The suspension space $\SSS$ is defined as the quotient
of $\Si\times \tl B$, 
where $\tl B$ is the universal covering of $B$,  by  the diagonal
action of $K$: $(\si ,y) \sim (\si +\rho (\alpha ),\alpha (y)) $ 
$ \forall \si \in \Si ,\; y \in \tilde{B} $ and $\alpha (y) $ 
being the application  of $\alpha \in K \cong \pi _1(B,b)$ to $y$.
Indeed, it is easy to see that if we choose a base point 
$ y \in \pi ^{-1}b \subset \tilde{B} $,
then the elements of $ p_\infty ^{-1}b \subset \SSS $ can be ``enumerated'' 
by elements of $\Sigma $, and that conditions (i) and (ii) in the definition
 of a flat
principal $\Si $-bundle are satisfied. 

Thus, the space $\CC(\Si)$ of principlal flat $\Si$-bundles over $B$
(mod a natural equivalence)
is identified with the space of representations $\rho: K \ra \Si$.

In the case of $B=X_\infty$ and $\Si = \Z _p$, where $\Z _p=\underleftarrow{\lim }
\:\Z /p^r $ is the group of p-adic numbers, the space $\CC(\Z _p) $ of 
flat principal $\Z _p$-bundles (mod natural equivalence) is identified with 
the space of $p$-adic 
representations $\Hom (K,\Z _p)$.
To the bundle
$$
\xymatrix {
\Z _p \ar[r] &
\SSS \ar[d]_{p_\infty}\\
& X_\infty 
}
$$
corresponding to a representation $\rho $, there are associated 
$ \Z /p^r $-bundles 
$$
\xymatrix {
\Z /p^r \ar[r] &
 S_r  \ar[d]_{p_r} \\
& X_\infty 
}
$$
corresponding to homorphisms $\rho _r :K\ra \Z /p^r$, where
$\rho _r$ is the composition 
$$
\xymatrix {
K \ar[r]^\rho & \Z _p \ar[r]^\pi &\Z /p^r
}
,$$
$\pi $ being the natural projection.
Clearly, $S_r $ form a tower of coverings and 
$\SSS= \underleftarrow{\lim }\:S_r$.

Note that $S_r $ is connected iff $\rho _r :K \ra \Z /p^r $ is 
onto. In the case when all $\rho _r $ are onto, $\SSS $ is a solenoid over
$X_\infty $. If for some $r$, $\:\rho _r$ is not onto, $S_r $ 
is disconnected.

The pullback action of the deck transformation $\tau$  on $\CC (\Z _p)$
corresponds to the $\si_x$-action in $\Hom(K, \Z_p)$.   

The latter  space is a finite dimensional $\Z_p$-module. Let us endow it with
the $\sup$-norm. Then any invertible operator 
$A: \Hom(K, \Z_p)\ra \Hom(K, \Z_p)$ becomes an isometry. 
Since $\Hom(K, \Z_p)$ is compact, $A$ is {\it almost periodic} in the
sense that the cyclic operator  group $\{A^n\}_{n\in \Z}$ is precompact.  The
closure of this group is called the {\it Bohr compactification}  of
$A$ (see \cite{Lyu}). Theorem \ref{lcm} provides us with a description
of this group for $\si_x$:

\begin{thm}
    The Bohr compactification of the operator 
$$
    \si_x:  \Hom(K,   \Z_p)\ra \Hom(K, \Z_p)
$$ 
is the inverse limit of the cyclic groups $\Z/d_n$
     where the $d_n$ are the least common multiplies described by Theorem \ref{lcm}.
\end{thm}

We can also consider solvable coverings over the knot complement $X$ described 
in 
\S 7. Taking their inverse limits, we obtain various
solenoids over $X$.

%% file: mainzpn.bbl
\begin{thebibliography}{*****}
\bibitem [K]{K} B.P. Kitchens,
``Expansive dynamics on zero-dimensional groups,''
Ergodic Theory and Dynamical Systems \boldmath$7$ (1987), 249-261. 
MR \boldmath$88i$:28039
\bibitem [L]{L} L. Lyubich,   
 ``Periodic orbits of a dynamical system related to a knot,``
Knot theory and its ramifications, v.20 N3 (2011), p411-426,

\bibitem[Lyu] {Lyu} Yu.I. Lyubich. Introduction to the theory of  Banach
  representations of groups.  Birkah\"auser.

\bibitem [S]{Si} D.S. Silver, ``Augmented group systems and $n-$ knots,''
Math.Ann. \boldmath${296}$ (1993), 585-593. MR \boldmath${94i}$:57039
\bibitem [SW1]{SW 1} D.S. Silver and S.G. Williams,  
 ``Augmented group systems and shifts of finite type,''
Israel J. Math. \boldmath${95}$, (1996), 231-251. MR \boldmath${98b}$:20045  
\bibitem [SW2]{SW 2} D.S. Silver and S.G. Williams, 
 ``Knot invariants from symbolic dynamical systems,''
Trans.Amer.Math.Soc. V351 N8 (1999), p3243-3265,S 0002-9947(99)02167-4
 

\end{thebibliography}
